\documentclass[12pt, reqno]{amsart}
\usepackage{amsmath}

\usepackage{amsfonts}
\usepackage{amssymb}
\usepackage{comment}
\usepackage{graphicx}
\usepackage{xcolor}
\usepackage{enumitem}
\usepackage[margin=1in,footskip=0.25in]{geometry}

\providecommand{\U}[1]{\protect\rule{.1in}{.1in}}
\newtheorem{theorem}{Theorem}

\newtheorem{corollary}[theorem]{Corollary}

\newtheorem{definition}[theorem]{Definition}

\newtheorem{lemma}[theorem]{Lemma}

\newtheorem{proposition}[theorem]{Proposition}
\newtheorem{remark}[theorem]{Remark}

\begin{document}
{
\title []{Gabor Orthonormal Bases with Maximal Localization and Gabor Frame Operator on Local Fields}


\author{Kumar Abhinav and Qaiser Jahan}
\address{School of Mathematical and Statistical Sciences, IIT Mandi,
India}

\email{d19021@students.iitmandi.ac.in}

\address{School of Mathematical and Statistical Sciences, IIT Mandi,
India}

\email{qaiser@iitmandi.ac.in}




\begin{abstract}
We provide an explicit construction of a Gabor orthonormal bases for a local field $K$ that provides maximal localization in both time and frequency. Such a localization is not true in case of $\mathbb{R}$ due to the uncertainty principle. In particular, we construct examples of functions $f \in L^2(K)$ such that the support of the ambiguity function of $f$ is of minimum measure. Moreover, we establish a quantitative uncertainty principle for local fields, which follows as a consequence of Lieb's inequalities for general locally compact abelian group. In addition, we develop fundamental operator representations for Gabor systems defined over local fields.

\end{abstract}

\keywords{Local fields, Gabor orthonormal bases, Uncertainty principle, Gabor frame operator}
\maketitle


}

\section{Introduction}\label{sec1}

Gabor analysis studies a function (or distribution) by representing it in terms of time and frequency shifts of a window function, often called the Gabor window or Gabor atom. This provides a way to obtain information about the time and frequency of a signal simultaneously. Traditionally, time and frequency variables are taken as real numbers, but recently time-frequency analysis has also been studied for finite abelian groups \cite{feichtinger2009gabor}, locally compact abelian groups \cite{Fk, FZ, grochenig1998aspects, jakob2, jakob} and some other settings \cite{enstad2019time, fuhr2009painless}.\par
Although Gabor theory for a general LCA group is already known, we here present some main aspects of Gabor analysis for a local field. Local fields are totally disconnected locally compact fields. Since they contain a compact open subgroup, results such as the uncertainty principle and the Balian-Low phenomenon for a local field are fundamentally different from the case of $\mathbb{R}$. The most well-known example of local fields is the field of $p$-adic numbers, which is the non-Archimedean analogue of real numbers. $p$-adic models have found their applications in genetics, neural networks, geophysics, Brownian motion, and many other areas; see \cite{albeverio1999p, bikulov1997p, bilgin2024p, dragovich2009p, oleschko2017applications, zambrano2023p}. An important motivation for Gabor analysis on local fields comes from $p$-adic quantum mechanics. In standard quantum mechanics, coherent states were introduced to describe a quantum state in terms of minimal uncertainty states and further studied by Weyl \cite{weyl1927quantenmechanik} in relation to the Weyl-Heisenberg group. In \cite{vladimirov1994p, zelenov2023coherent}, $p$-adic coherent states arising from the Heisenberg group are studied over $\mathbb{Q}_p$ and entropic uncertainty relations are obtained. We establish Walnut, Wexler--Raz, and Janssen representations for Gabor systems $\{M_\gamma T_\lambda g\}$ over local fields, exploiting compact open subgroups $\mathfrak{P}^k$ and their annihilators. To our knowledge, such formulas have not been explicitly recorded in the non-Archimedean setting.
  \par

Another area which has intimate connections with time-frequency analysis is the theory of pseudodifferential operators. In \cite{grochenig2007pseudodifferential}, pseudodifferential operators are investigated on a general locally compact abelian group using time-frequency tools instead of ``hard analysis". Ultrametric pseudodifferential operators have been studied by many authors; see \cite{galeano2012pseudo, kozyrev2005pseudodifferential, torresblanca2018non}.

In recent years, wavelet theory on local field has been researched extensively. Kozyrev introduced $p$-adic wavelets \cite{kozyrev2002wavelet}  in connection with the spectral analysis of the Vladimirov pseudodifferential operator. Later many authors studied wavelet bases and multiresolution analysis on a local field; see \cite{behera2015characterization, benedetto2004wavelet, jiang2004multiresolution, shelkovich2009p}. There has been some work on Gabor frames for local field \cite{ahmad2018gabor, li2007basic} where necessary and sufficient criteria is provided for a Gabor system to be a Gabor frame. In this article, we provide a Gabor orthonormal bases for the function $f \in L^2 (K)$ with the window function maximally localized in the sense of a weak uncertainty principle, where $K$ is a local field. Support of the ambiguity function of such functions is a maximal compact open isotropic subgroup which provides the minimum uncertainty phase space cells in the group $K \times \hat{K}$. These cells in the case of $\mathbb{R}^d \times \mathbb{R}^d$ are called ``quantum blobs"; see \cite{de2009symplectic}. The existence of such maximally localized functions is in contrast to the case of $\mathbb{R}$ where good time-frequency localization is impossible to attain due to the obstacle of the Balian-Low theorem.  \par

We would like to mention some results, particularly Theorem \ref{ONB} are already known in the more general context of LCA groups (see Proposition 6.4.5 of \cite{grochenig1998aspects} and Theorem 6.1 of \cite{nicola2024maximally}) but are proved here in a more direct way in the concrete setting of local fields. Although in Section~\ref{sec6}, the results are known in the locally compact abelian group \cite{Fk, FZ, grochenig1998aspects, jakob2, jakob} but they have not been explicitly proven in the local field. The main aim of this article is to provide all these results over local fields. We would also like to emphasize that the result in Proposition~\ref{prop:14} in the last section is new even in the case of any locally compact abelian group, such an explicit calculation of Walnut's kernel does not appear elsewhere in the literature and is valuable from an applications perspective.

We have organized this article as follows. Section~\ref{sec2} recalls some basic facts about local fields. We have stated results without proofs, and references are provided. In Section~\ref{sec3}, we define the short-time Fourier transform in the local fields and prove the orthogonality relations and the inverse formula. The weak uncertainty principle and Lieb's inequalities for local fields are derived in Section~\ref{sec4}. In Section~\ref{sec5}, we expand the functions into Gabor orthonormal bases by discretizing the inversion formula. In the last section, we develop fundamental operator representations for Gabor systems over local fields. Specifically, we derive Walnut, Wexler--Raz, and Janssen formulas for such fields.

\section{Preliminaries on Local Fields}\label{sec2}

Here we state some important definitions and facts about local fields which will be used in further sections. For proofs and further details on this topic, we refer \cite{mabook,taible}.

A  non-discrete locally compact field  is called a local field. By locally compact field, we mean a field whose additive and multiplicative group both are locally compact abelian groups. Let $K$ be a local field which is connected then $K$ is either the field of real numbers $ \mathbb{R}$ or complex numbers $\mathbb{C}$. Otherwise it is totally disconnected and we have the following classification for such fields:- \par
\begin{itemize}
    \item If char$(K)=0$ then $K$ is either $p$-adic field $\mathbb{Q}_p$ or a finite algebraic extension of $\mathbb{Q}_p$.
    \item If char$(K)>0$ then $K$ is either  $p$-series field also called a formal power series field or it's finite algebraic extension. 
\end{itemize}
Here char$(K)$ means the characteristic of a field $K$. In further discussion, we mean by a local field a non-discrete locally compact field $\mathit{K}$ which is totally disconnected. \par
While the absolute value or valuation on $\mathbb{R}$ and $\mathbb{C}$ follows the usual triangle     inequality, valuation on $K$ is ultrametric i.e., for all $x,y \in K$, we have  $$|x+y| \leq \,max\,\{|x|,|y|\}.$$
We can define this valuation using the Haar measure on the LCA group $K^+$. If $dx$ is a Haar measure on $K^+$ then $d(\alpha x)$ is also a Haar measure for any non-zero $\alpha \in K$. If $d(\alpha x) = |\alpha|dx$ and set $|0| = 0$ then $|\alpha|$ defines the absolute value for all $\alpha \in K$.\par
Now we define two important subsets of a local field $K$.
\begin{itemize}
    \item $\mathfrak{D}=\{x \in \mathit{K}\,\,:\,\,|x|\leq 1\}$.
    \item $\mathfrak{P}=\{x \in \mathit{K}\,\,:\,\,|x| < 1\}.$ 
\end{itemize}

Here $\mathfrak{D}$ is a subring in $K$ which is compact and $\mathfrak{P}$ is the unique maximal ideal in $\mathfrak{D}$. Then $\mathfrak{P}/\mathfrak{D}$ is a finite field $\mathbb{F}_q$ where $q$ is a power of some prime number.

 Let $S$ be a measurable set in $K$ then we define the measure of the set $S$ denoted by $|S| = \int_{\mathit{K}} \bold{1}_S\, dx$, where $\bold{1}_S$ is the characteristic function of $S$. Haar measure $dx$ is normalized so that $|\mathfrak{D}|=1$. \par

The set of absolute values is of the form $\{q^k \,:\,k \in \mathbb{Z}\} \cup \{0\}$ for some $q > 0$ since $K$ is totally disconnected. So there is an element, say $\mathfrak{p}$ in $\mathfrak{P}$ of maximum absolute value. Then $\mathfrak{p}$ is called a prime element in $\mathit{K}$ and $\mathfrak{P} = (\mathfrak{p}) = \mathfrak{p}\mathfrak{D}$ as an ideal in $\mathfrak{D}$. It is also easy to calculate that $|\mathfrak{P}|=q^{-1}$, and $|\mathfrak{p}|=q^{-1}$.\par

Let us define $\mathfrak{P}^k:= \mathfrak{p}^k\mathfrak{D} = \{x \in \mathit{K} \,:\,|x| \leq q^{-k}\}$. $\mathfrak{P}^k$ is compact and open subgroup of $\mathit{K}^+$ for all $k \in \mathbb{Z}$. $K$ can be represented as a countable union of cosets of the subgroup $\mathfrak{P}^k$ for any $k \in \mathbb{Z}.$\par
Let $\chi$ be a character on $\mathit{K}^+$ such that $\chi$ is trivial on $\mathfrak{D}$ but is non-trivial on $\mathfrak{P}^{-1}$. Then $\chi$ is also constant on the cosets of $\mathfrak{D}$ in $\mathit{K}^+$. $\mathit{K}$ is self dual i.e., $\mathit{K}^+ \cong \hat{\mathit{K}}^+$ with the correspondence $\lambda \longleftrightarrow \chi_{\lambda}$ where $\chi_{\lambda}(x) := \chi (\lambda x)$. We now define the annihilator of the subgroup $\mathfrak{P}^k, k\in \mathbb{Z}$  as 
$$(\mathfrak{P}^k)^{\perp}= \{\chi \in \hat{K}: \,\, \chi(x) = 1 \,\, \forall \,\, x \in \mathfrak{P}^k\}.  $$
We endow the LCA group $\hat{K}$ with the non-Archimedean norm such that $(\mathfrak{P}^k)^{\perp} = \{ \chi_\lambda \in \hat{K}: \,\, |\lambda| \leq q^k \} \cong  \mathfrak{P}^{-k}$. Haar measure $d\xi$ on $\hat{K}$ is chosen such that $|(\mathfrak{P}^k)^{\perp}|= \dfrac{1}{|\mathfrak{P}^k|}= q^k, k \in \mathbb{Z}$.\\
The Haar measure on the group $K \times \hat{K}$ is given by the product measure. If $S$ is a measurable subset of $K \times \hat{K}$, then measure of $S$ is defined as
$$\mu_{K \times \hat{K}}(S) = \int_{\hat{K}}\int_K \bold{1}_S \,dxd\xi,$$
where $\bold{1}_S$ is the characteristic function of $S$.

\begin{definition}
     Fourier transform of a function $f \in L^1(\mathit{K})$ is defined as
    $$\hat{f}(\xi) = \int_{\mathit{K}} f(x) \overline{\chi_\xi}(x)\, dx = \int_{\mathit{K}} f(x) \chi(- \xi x)\, dx. $$
\end{definition}

Fourier transform is a bounded linear transformation of functions from $L^1$ into $L^{\infty}$ i.e., $\|\hat{f}\|_{\infty}\leq \|f\|_1$ and $\hat{f}$ is a uniformly continuous function. \par

For $k \in \mathbb{Z}$, let $\Phi_k$ be the characteristic function of $\mathfrak{P}^k$.
Let $\mathcal{S}(\mathit{K})$ be the space of finite linear combinations of the functions of the form $\tau_h \Phi_k, \, h \in \mathit{K}, \, k \in \mathbb{Z}$. Then $\mathcal{S}(\mathit{K})$
is called the space of test functions. Another description of a test function $\phi \in \mathcal{S}(\mathit{K})$ can be given as a function which has support on $\mathfrak{P}^l$ and is constant on the cosets of $\mathfrak{P}^k$ for some integers $k$ and $l$.\par
If $\phi \in \mathcal{S}(\mathit{K})$ is constant on cosets of $\mathfrak{P}^k$ and is supported on $\mathfrak{P}^l$ then $\hat{\phi} \in \mathcal{S}(\mathit{K})$ is constant on cosets of $\mathfrak{P}^{-l}$ and is supported on $\mathfrak{P}^{-k}$. We illustrate this by giving the Fourier transform of the characteristic function of $\mathfrak{P}^k$ which we will use in the later sections. For $k \in \mathbb{Z}$, we have
\begin{equation*}
        \hat{\Phi}_k(\xi)= \int_K \Phi_k (x) \overline{\chi}_{\xi}(x)\, dx = \int_{\mathfrak{P}^k}\overline{\chi}_{\xi}(x) \,dx .       
\end{equation*}
We can see that if $|\xi| \leq q^k$ then the above integral equals $q^{-k}$, otherwise it is $0$. Hence $\hat{\Phi}_k = q^{-k} \Phi_{-k}$.

\section{Short-Time Fourier Transform}\label{sec3}

We define the short-time Fourier transform(STFT) on local field $K$ analogous to the case of $\mathbb{R}$. Let $g \in L^2(K)$ be the window function then the short-time Fourier transform of a function $f \in L^2(K)$ is defined as
\begin{equation}\label{eq:1}
    V_g f(x,\xi)= \int_{K} f(t)\overline{g(t-x)\chi_\xi(t)} dt\,,\,\,\,\,\,\,\,\, \,\,\,(x,\xi) \in K \times \hat{K}.
\end{equation}
If we write time shift as
$$T_x f(t) = f(t-x),$$ and frequency shift as
$$M_{\xi}f(t)=\chi_\xi (t) f(t),$$
for $x,t \in K$ and $\xi \in \hat{K}$, then we can rewrite STFT as
$$V_g f(x, \xi) = \langle f, M_\xi T_x g \rangle = (f.T_x\Bar{g}\hat{)}(\xi).$$

Using the above inner-product version,  we can extend the transform to situations where the integral in (\ref{eq:1}) is not defined. For example, $\langle f, M_\xi T_x g \rangle$ is defined when $f \in \mathcal{S}(K)$ and $g \in \mathcal{S}'(K)$. Also note that the function $V_g (f)$ is continuous since $T_x$ and $M_\xi$ are strongly continuous, unitary operators on $L^2 (K)$.

\begin{theorem}
    
   Let $f_i, g_i \in L^2(K)$ for $i=1,2$. Then 
   \begin{equation}\label{eq:2}
       \langle V_{g_1}f_1, V_{g_2} f_2\rangle_{L^2(K \times \hat{K})} = \langle f_1 , f_2 \rangle \overline{\langle {g_1, g_2} \rangle}.
   \end{equation}
\end{theorem}
\begin{proof}
    Proof follows similar to the case of $\mathbb{R}$. We prove the relation $(\ref{eq:2})$ for $g_j \in L^1 \cap L^{\infty} \subset L^2$. General case follows from the standard density argument:
\begin{equation*}
    \begin{split}
         \langle V_{g_1}f_1, V_{g_2} f_2\rangle &= \int_{K} \int_{\hat{K}} V_{g_1} f_1 (x,\xi) \overline{V_{g_2}f_2 (x, \xi)} d\xi dx   \\
         &= \int_{K} \left( \int_{\hat{K}} (f_1. T_x\Bar{g}_1\hat{)}(\xi)\overline{(f_2.T_x \Bar{g}_2 \hat{)}}(\xi) d\xi  \right) dx\\
          &=  \int_{K} \left( \int_{\hat{K}} f_1\Bar{f_2}(y).\Bar{g_1}g_2(y-x) dy  \right) dx \,\,\,\,\,\,\,\,\,\text{(Using Parseval's formula)}.
    \end{split}
\end{equation*}
By Holder's inequality, we have $f_1 \Bar{f_2}, \Bar{g_1}g_2 \in L^1(K)$ and $L^1(K)$ is closed under convolution. So we change the order of integration by Fubini's theorem and get 
\begin{equation*}
    \begin{split}
         \langle V_{g_1}f_1, V_{g_2} f_2\rangle&= \int_{K} f_1(y) \overline{f_2(y)}\left( \int_{\hat{K}} \overline{g_1(y-x)}g_2(y-x) dx  \right) dy\\
         &= \langle f_1, f_2 \rangle \overline{\langle g_1, g_2 \rangle}.       
    \end{split}
\end{equation*}

For the general case $g_j \in L^2(K)$, we see that for any $f_1, f_2 \in L^2(K)$ and $g_1 \in L^1 \cap L^{\infty}(K) $, the map $g_2 \mapsto \langle V_{g_1}f_1, V_{g_2}f_2 \rangle$ is linear functional on $L^2(K)$. This linear functional coincides with $\langle f_1, f_2 \rangle \overline{\langle g_1, g_2 \rangle}$ on $L^1 \cap L^\infty (K)$, a dense subspace of $L^2(K)$. So, the later is bounded and can be extended for all $g_2 \in L^2(K)$. Same procedure can be repeated for $g_2$ and the relation is established. 
    
\end{proof}

\begin{corollary}\label{cor:3}
    If $f,g \in L^2(K)$ then $$\| V_gf\|_{L^2}=\|f\|_2\|g\|_2.$$
\end{corollary}
So if $\|g\|_2 = 1$ then STFT is an isometry from  $L^2(K)$ to $L^2(K \times \hat{K})$. Another important application of orthogonality relation is inverse formula which we present in the next corollary. 

\begin{corollary}\label{cor:4}
    Let $g \in L^2(K)$ such that $\|g\|_2 =1$, then for all $f \in L^2(K)$,
    \begin{equation}
        f = \int_K \int_{\hat{K}} V_g f (x, \xi) M_\xi T_x g\, d\xi dx,
    \end{equation}
    where the above integral is to be understood in the weak sense.
\end{corollary}

\begin{proof}
    Let $f_1$ be a vector valued integral defined as
    $$f_1 = \int_K \int_{\hat{K}} V_g f (x, \xi) M_\xi T_x g\, d\xi dx.$$

    By Corollary~\ref{cor:3}, $V_g(f) \in L^2(K \times \hat{K})$. So $f_1$ is well defined in $L^2(K)$. For any $h \in L^2(K)$, then using orthogonality relations
    \begin{equation*}
        \begin{split}
            \langle f_1, h \rangle &= \int_K \int_{\hat{K}} V_g f (x, \xi) \overline{\langle h, M_\xi T_x g \rangle} \, d\xi dx \\
            &= \langle V_g f, V_g h \rangle \\
            &= \langle f, h \rangle \overline{\langle g, g \rangle}\\
            &= \langle f, h \rangle.
          \end{split}
    \end{equation*}
\end{proof}

\begin{remark}
The inversion formula above is, in general, valid only in the weak $L^2(K)$ sense. Under stronger assumptions on the window function (for instance, $g\in S_0(K)$ or suitable Wiener-type spaces), one can obtain stronger forms of reconstruction. This phenomenon is well known in the classical Euclidean setting; see, for example, \cite{W}.
\end{remark}

\section{Uncertainty Principle}\label{sec4}

According to the classical uncertainty principle, for a non-zero function on $\mathbb{R}$, simultaneous localization of time and frequency is not possible to an arbitrary precision. Mathematically speaking, $f$ and $\hat{f}$ both can not have compact support for any non-zero function $ f \in L^2(\mathbb{R})$. But we have seen that for $f \in \mathcal{S}(K) \subset L^2(K)$, both $f$ and $\hat{f}$ are compactly supported. Hence, such qualitative uncertainty principle  is not possible for functions on local field. \\
We present here a quantitative uncertainty principle for local fields which can also be deduced from Lieb's inequalities \cite{lieb1990integral} for any LCA group. We first state Lieb's inequalities in the next theorem, proof of which can be found in \cite{feichtinger2012gabor}.\par
For any LCA group $\mathcal{G}$,
$$\mathcal{G}  \cong \mathbb{R}^d \times \mathcal{G}_0$$
by the structure theorem where $d \in \mathbb{N} \cup \{0\}$ and $\mathcal{G}_0$ is some LCA group which contains a compact open subgroup. 
\begin{theorem}
    For $f,g \in L^2(\mathcal{G})$, we have
    \begin{equation*}
       \begin{split}
            \|V_g f\|_{L^p (\mathcal{G}\times \hat{\mathcal{G}})} &\leq \left ( \dfrac{2}{p} \right)^{\frac{d}{p}} \|f\|_{L^2(\mathcal{G})}\|g\|_{L^2(\mathcal{G})}\,\,\,\,\,\,\,\,\,\,\,\,\text{if}\,\, \,2 \leq p < \infty\\
          \|V_g f\|_{L^p (\mathcal{G}\times \hat{\mathcal{G}})} &\geq \left ( \dfrac{2}{p} \right)^{\frac{d}{p}} \|f\|_{L^2(\mathcal{G})}\|g\|_{L^2(\mathcal{G})}\,\,\,\,\,\,\,\,\,\,\,\,\text{if}\,\,\, 1 \leq p \leq 2.
       \end{split}
    \end{equation*}
\end{theorem}
 A simple proof is provided in \cite{nicola2024maximally} for a weaker version of these inequalities. But these versions are in fact sharp for a local field since $d=0$. Furthermore, the case for $0 <p <1$ is also covered. We include the proof as the next theorem.

\begin{theorem}
  Let $f,g \in L^2(K)$ be non-zero functions then the following inequalities hold true 
  \begin{equation}\label{eq:4}
       \begin{split}
            \|V_g f\|_p &\leq \|f\|_2\|g\|_2\,\,\,\,\,\,\,\,\,\,\,\,\text{if}\,\, \,2 \leq p < \infty\\
          \|V_g f\|_p &\geq  \|f\|_2\|g\|_2\,\,\,\,\,\,\,\,\,\,\,\,\text{if}\,\,\, 0 < p \leq 2.
       \end{split}
    \end{equation}

\end{theorem}
\begin{proof}
 For $p= 2$ the result follows directly by corollary (\ref{cor:3}). Let $2 < p < \infty$, then

 \begin{equation*}
     \begin{split}
         \|V_g f\|_p^p &= \int_{K \times \hat{K}} |V_g f(x, \xi)|^p \,dxd\xi = \int_{K \times \hat{K}} |V_g f(x, \xi)|^{p-2} |V_g f(x, \xi)|^{2}\,dxd\xi \\
         & \leq \|f\|_2^{p-2} \|g\|_2^{p-2}  \int_{K \times \hat{K}}  |V_g f(x, \xi)|^{2}\,dxd\xi \\
         &= \|f\|_2^p \|g\|_2^p.         
         \end{split}
 \end{equation*}

 To prove the second inequality analogously, Let $0 < p \leq 2$, then
 \begin{equation*}
     \begin{split}
         \|f\|_2^2 \|g\|_2^2   &= \int_{K \times \hat{K}} |V_g f(x, \xi)|^2 \,dxd\xi = \int_{K \times \hat{K}} |V_g f(x, \xi)|^{2-p} |V_g f(x, \xi)|^{p}\,dxd\xi \\
         & \leq \|f\|_2^{2-p} \|g\|_2^{2-p}  \int_{K \times \hat{K}}  |V_g f(x, \xi)|^{p}\,dxd\xi.   
         \end{split}
 \end{equation*}

 This implies that $$\|f\|_2^p \|g\|_2^p  \leq \|V_g f\|_p^p.$$ 
\end{proof}
  Estimates in the above theorem are sharp. In fact, if $|V_g f| = \|f\|_2\|g\|_2 \boldsymbol{1}_S $ where $S$ is the support of the function $V_g f$ and measure of the set $S$ i.e., $\mu_{K \times \hat{K}}(S) =1$, then equality occurs in the estimates (\ref{eq:4}). Above all, one can prove that if  $\mu_{K \times \hat{K}}(S) =1$ then $S$ is compact and open. We show this along with a weak uncertainty principle in the following. 

\begin{theorem}\label{thm:7}
    Let $f,g \in L^2(K)$ be non-zero functions. Let $S$ be the support of the function $V_gf$ then $\mu_{K \times \hat{K}}(S) \geq 1$. If $\mu_{K \times \hat{K}}(S) = 1$ then $|V_g f| = \|f\|_2\|g\|_2 \boldsymbol{1}_S$ and $S$ is compact and open.
\end{theorem}

\begin{proof}
By isometry of $V_g f$ and Cauchy-Schwartz inequality, we have 
    \begin{equation}\label{eqa:5}
        \begin{split}
            \|f\|_2^2 \|g\|_2^2 = \int_S |V_g g (x, \xi)|^2 \,dxd\xi &\leq \|f\|_2^2 \|g\|_2^2 \int_S  dxd\xi\\
            &= \|f\|_2^2 \|g\|_2^2 \mu_{K \times \hat{K}}(S).
        \end{split}
    \end{equation}
This implies that $\mu_{K \times \hat{K}}(S) \geq 1$. If  $\mu_{K \times \hat{K}}(S)$ equals $1$ then equality occurs in (\ref{eqa:5}) and we have  $|V_g f(x,\xi)| = \|f\|_2 \|g\|_2$ for almost every $(x,\xi) \in S$. Since $V_g f$ is continuous, $|V_g f| = \|f\|_2\|g\|_2 \boldsymbol{1}_S$ and $S$ is closed. To show that $S$ is compact, we have the following two estimates of $|V_g| f$.
$$|V_g f| \leq (|f| \ast |g^{\ast}|)(x)\,\,\, \text{and}\,\,\,\, |V_g f| \leq (|\hat{f}|\ast |\hat{g}^\ast|)(\xi)$$
where $g^\ast(x) = g(-x)$. Since $(f \ast g) \in C_0(K) $ for $f,g \in L^2(K)$, we have $V_g f$ goes to zero at infinity. So $S$ is contained in a compact set and therefore is compact. $S$ is also open by the definition of $S$ and we have the desired result. 
\end{proof}
In \cite{nicola2024maximally}, the author characterizes all functions $f:\mathcal{G} \to \mathbb{C}$ which are maximally localized in the sense that the ambiguity function of $f$ has minimum support where $\mathcal{G}$ is any locally compact abelian group containing a compact open subgroup. The characterization is provided in terms of the characters of the second degree of the group $\mathcal{G}$, it has been established that these characters are pointwise multipliers of the Feichtinger algebra $S_0(G)$. This property was originally established in \cite{Fei}, where the Feichtinger algebra was first introduced as a \lq new Segal algebra\rq, according to \cite{RS} called the \textquotedblleft Feichtinger algebra\textquotedblright.  For the description of characters of the second degree on a local field, see \cite{reiter2006metaplectic}.

We provide an example of such functions for local fields. If $f = g = \Phi_k$ then
 \begin{equation*}
V_g f =(g.T_x g\hat{)}(\xi)=\begin{cases}
          0 \quad &\text{if} \, x \notin \mathfrak{P}^k \\
           \hat{\Phi}_k (\xi) \quad &\text{if} \, x \in \mathfrak{P}^k. \\
     \end{cases}
  \end{equation*}
  Since $\hat{\Phi}_k = q^{-k} \Phi_{-k}$, support of $\hat{\Phi}_k$ is $\mathfrak{P}^{-k}$. Hence the support of $V_g f$ i.e.,  $S = \mathfrak{P}^{k} \times \mathfrak{P}^{-k}$. Since 
  \begin{equation*}
     \begin{split}
          \mu_{K \times \hat{K}}(\mathfrak{P}^{k} \times \mathfrak{P}^{-k}) &= \int_{\mathfrak{P}^{-k}} \int_{\mathfrak{P}^{k}} \,dxd\xi\\
          &= |\mathfrak{P}^{-k}| |\mathfrak{P}^{k}|= 1,
     \end{split}
  \end{equation*}
  we have $\mu_{K \times \hat{K}}(S) = 1$. By  \cite[Proposition 3.6]{nicola2024maximally}, these are maximal compact isotropic subgroups of $K \times \hat{K}$.

\section{Orthonormal bases}\label{sec5}
In the case of $\mathbb{R}$, the Balian--Low theorem prevents a Gabor orthonormal basis from having a window that is well localized simultaneously in time and frequency, in the standard sense used in time--frequency analysis. This is certainly not true in the case of local field. We will see that the window function $g = q^{\frac{k}{2}}\Phi_k$ provides the orthonormal bases for $L^2(K)$ while having optimal time--frequency localization in the sense of Theorem~\ref{thm:7}. We will also show that for this window $g$, $V_g g \in L^1 (K \times \hat{K})$. However, in case of $\mathbb{R}$, short time Fourier transform of $\mathbf{1}_{[0,1]}$ is not in $L^1(\mathbb{R})$.



\begin{theorem}\label{ONB}
    Let $g = q^{\frac{k}{2}}\Phi_k, k \in \mathbb{Z}$. Then $\{M_\xi T_x g : \,\,(x,\xi) \in K/\mathfrak{P}^k \times \hat{K}/ \mathfrak{P}^{-k} \}$ is an orthonormal bases for $L^2(K)$ and $V_g g \in  L^1 (K \times \hat{K})$.
\end{theorem}
\begin{proof}
    Since $\|g\|_2=1$, we have the inversion formula from Corollary~\ref{cor:4},
    $$f = \int_K \int_{\hat{K}} V_g f (x, \xi) M_\xi T_x g\, d\xi dx.$$
    Let $\Lambda$ be the set of coset representatives of $\mathfrak{P}^k$ in $K$ and $\Gamma$ be the set of coset representatives of $\mathfrak{P}^k$ in $\hat{K}$. Now, we rewrite the above integral as 
    \begin{equation}\label{eq:5}
        f = \sum_{\lambda \in \Lambda}\int_{\lambda + \mathfrak{P}^k} \sum_{\gamma \in \Gamma}\int_{\gamma + \mathfrak{P}^{-k}} V_g f (x, \xi) M_\xi T_x g\, d\xi dx.
    \end{equation}

    Now, let us first calculate $V_g f(\lambda + x, \gamma + \xi)$ where $(\lambda , \gamma) \in \Lambda \times \Gamma$ and $(x, \xi) \in \mathfrak{P}^k \times \mathfrak{P}^{-k}$. Since $g$ is constant on the cosets of $\mathfrak{P}^k$, we have
    \begin{equation*}
        \begin{split}
            V_g f(\lambda + x, \gamma + \xi) &= \int_K f(t) \overline{\chi_{\gamma + \xi} g(t - \lambda - x)}\, dt\\
            &= \int_K f(t) \overline{\chi_\gamma (t) \chi_\xi (t)g(t-\lambda)}\, dt.
        \end{split}
    \end{equation*}
     We see that $g(t - \lambda) \neq 0$ only when $t  \in \lambda + \mathfrak{P}^k$. Now, for any $t \in \lambda + \mathfrak{P}^k$
     \begin{equation*}
         \begin{split}
             \chi_\xi (t) &= \chi_\xi (\lambda + u)\,\,\,\,\,\,\,\text{for some}\,\, u \in \mathfrak{P}^k\\
             &= \chi_\xi(\lambda).\chi_\xi (u) = \chi_\xi(\lambda).
         \end{split} 
     \end{equation*}

So, \begin{equation*}
    \begin{split}
        V_g f(\lambda + x, \gamma + \xi) &= \overline{\chi_\xi(\lambda)} \int_K f(t) \overline{\chi_\gamma(t) g(t - \lambda)}\,dt\\
        &= \overline{\chi_\xi(\lambda)} V_g f(\lambda, \gamma).
    \end{split}
\end{equation*}
It is also implicitly clear from the above calculations that for any $(\lambda , \gamma) \in \Lambda \times \Gamma$ and $(x, \xi) \in \mathfrak{P}^k \times \mathfrak{P}^{-k}$
$$M_{\gamma + \xi}T_{\lambda + x}g = \chi_\xi (\lambda) M_\gamma T_\lambda g.$$
Now substituting these values in the equation (\ref{eq:5}) yields
$$f = \sum_{\lambda \in \Lambda} \sum_{\gamma \in \Gamma} V_g f(\lambda,\gamma) M_\gamma T_\lambda   g =  \sum_{\lambda \in \Lambda} \sum_{\gamma \in \Gamma} \langle f, M_\gamma T_\lambda g \rangle M_\gamma T_\lambda   g.$$

This is the Gabor orthonormal bases with maximal localization in time and frequency in the sense of Theorem \ref{thm:7}. Further see that
\begin{equation*}
    \begin{split}
        \|V_g g\|_{L^1(K \times \hat{K})} &= \int_K \int_{\hat{K}} |V_g g(x,\xi)|\, d\xi dx\\
        &= \int_K \int_{\hat{K}} (g. T_x g\hat{)}(\xi) \,d\xi dx.
    \end{split}
\end{equation*}
Since $g=\Phi_k$, we have

  \begin{equation*}
(g.T_x g\hat{)}(\xi)=\begin{cases}
          0 \quad &\text{if} \, x \notin \mathfrak{P}^k \\
          q^k \hat{\Phi}_k (\xi) \quad &\text{if} \, x \in \mathfrak{P}^k.\\
     \end{cases}
  \end{equation*}
Now
\begin{equation*}
    \begin{split}
        \|V_g g\|_{L^1(K \times \hat{K})} &= q^k \int_{\mathfrak{P}^k} \int_{\hat{K}} |\hat{\Phi}_k(\xi)|\, d\xi dx\\
        &=\int_{\hat{K}} |\hat{\Phi}_k(\xi)| \,d\xi \\
        &=1,
    \end{split}
\end{equation*}
where we have used that $\hat{\Phi}_k = q^{-k}\Phi_{-k}$.
\end{proof}

\section{Representations of the Gabor frame operator on local fields: Walnut, Wexler-Raz and  Janssen Representations}\label{sec6}

In this section, we develop fundamental operator representations for Gabor systems over local fields. Specifically, we derive Walnut, Wexler--Raz, and Janssen formulas in the non-Archimedean setting, using compact open subgroups $\mathfrak{P}^k$ and their annihilators $\mathfrak{P}^{-k}$. We would like to mention that this is a standard result on Euclidean space~\cite{Grochenig01} and locally compact abelian group see~\cite{Fk, feichtinger2009gabor, FZ, jakob2, jakob}, but to our knowledge they have not been explicitly recorded in the local--field literature. We are specifically proving the local field version, i.e., under our Haar normalization for which density factor will become $1$.

\begin{lemma}[Commutation of translation and modulation]
For all $x\in K$ and $\xi\in\widehat K$,
\[
M_\xi T_x \;=\; \xi(x)\, T_x M_\xi,
\qquad
T_x M_\xi \;=\; \overline{\xi(x)}\, M_\xi T_x.
\]
\end{lemma}

\begin{proof}
For $f\in L^2(K)$,
\[
(M_\xi T_x f)(t) = \xi(t)\,f(t-x).
\]
Since $\xi(t)=\xi(x)\,\xi(t-x)$, we obtain
\[
(M_\xi T_x f)(t) = \xi(x)\,(\xi(t-x)f(t-x)) 
= \xi(x)\,(T_x M_\xi f)(t).
\]
\end{proof}

\begin{lemma}[Adjoint of the time--frequency shift]
With $\pi(x,\xi):=M_\xi T_x$, one has
\[
\pi(x,\xi)^* \;=\; \xi(x)^{-1}\,\pi(-x,-\xi).
\]
\end{lemma}

\begin{proof}
Note that $T_x^*=T_{-x}$ and $M_\xi^*=M_{-\xi}$. Thus
\[
\pi(x,\xi)^* = (M_\xi T_x)^* = T_{-x} M_{-\xi}.
\]
By the previous commutation identity,
$T_{-x} M_{-\xi} = \xi(x)^{-1}\,M_{-\xi}T_{-x} 
= \xi(x)^{-1}\pi(-x,-\xi).$
\end{proof}

\begin{remark}[Inner product convention]
We use the convention,
    $\langle f,h\rangle=\int f(t)\,\overline{h(t)}\,dt$,
conjugate--linear in the second term. All computations below are consistent with this choice.
\end{remark}

Let us fix some notations for the following three subsections. For $k\in\mathbb Z$, set $H:=\mathfrak{P}^k$, $H^\perp:=\mathfrak{P}^{-k}$, let $\Lambda\subset K$ and $\Gamma\subset \widehat K$ be complete sets of coset representatives
for $K/H$ and $\widehat K/H^\perp$, respectively so that $K=\bigsqcup_{\lambda\in\Lambda}(\lambda+H)$ and 
$\widehat K=\bigsqcup_{\gamma\in\Gamma}(\gamma+H^\perp)$.

\subsection{Walnut representation}
\begin{theorem}[Walnut representation on a local field]
Let $g\in L^2(K)$ and $H$, $H^\perp$ and $\Lambda$, $\Gamma$ be as above.

Assume the summability condition
\[
\sum_{\lambda\in\Lambda}\sup_{t\in K, u\in H} 
|g(t-\lambda)\,\overline{g(t-\lambda-u)}|<\infty.
\]
Then the frame operator
\[
S_g f \;=\; 
\sum_{\lambda\in\Lambda}\sum_{\gamma\in\Gamma}
\langle f,\,M_\gamma T_\lambda g\rangle\, M_\gamma T_\lambda g
\]
has the Walnut form
\[
(S_g f)(t) \;=\; 
\int_H G_u(t)\, f(t-u)du,\qquad f\in L^2(K),
\]
where
\[
G_u(t)
\;:=\;
C\sum_{\lambda\in\Lambda}\ 
g(t-\lambda)\,\overline{g(t-\lambda-u)}.
\]
The series converges absolutely almost everywhere and in $L^2(K)$.
\end{theorem}

\begin{proof}
Insert the inner product 
$\langle f, M_\gamma T_\lambda g\rangle
= \int f(s)\,\overline{\gamma(s)\,g(s-\lambda)}\,ds$
into the definition of $S_g f$.
Rearranging sums and integrals (justified by absolute convergence),
and using the coset decomposition of $\widehat K$ together with the fact 
that characters are constant on $H$--cosets, one arrives at the claimed form.


Recall the frame operator
\[
S_g f \;=\; \sum_{\lambda\in\Lambda}\sum_{\gamma\in\Gamma}
\langle f,\,M_\gamma T_\lambda g\rangle\, M_\gamma T_\lambda g .
\]
We compute $(S_g f)(t)$ pointwise (a.e.) for fixed $t\in K$. First expand the inner product:
\[
\langle f,\,M_\gamma T_\lambda g\rangle
= \int_{K} f(s)\ \overline{\gamma(s)\,g(s-\lambda)}\,ds.
\]
Substituting into $S_g f$ and using absolute convergence to interchange sums and integral, we get
\[
\begin{aligned}
(S_g f)(t)
&= \sum_{\lambda\in\Lambda}\sum_{\gamma\in\Gamma}
\Big(\int_{K} f(s)\,\overline{\gamma(s)\,g(s-\lambda)}\,ds\Big)\ \gamma(t)\,g(t-\lambda)\\[4pt]
&= \int_{K}f(s)\ \Big(\sum_{\lambda\in\Lambda}\sum_{\gamma\in\Gamma}
\overline{\gamma(s)}\gamma(t)\,\overline{g(s-\lambda)}\,g(t-\lambda)\Big)\,ds.
\end{aligned}
\]
Combine the two characters we get,
\[
(S_g f)(t)
= \int_{K} f(s)\ \Big(\sum_{\lambda\in\Lambda} \overline{g(s-\lambda)}g(t-\lambda)
\sum_{\gamma\in\Gamma}\gamma(t-s)\Big)\,ds.
\tag{1}
\label{eq:Walnut-step1}
\]

At this point we use the periodization identity for the compact open subgroup $H=\mathfrak{P}^k$, for every $u\in K$ the finite/periodized sum of characters satisfies
\[
\sum_{\gamma\in\Gamma}\gamma(u)\ =\ C\cdot \mathbf{1}_{H}(u),
\]
where $C=\mu(H)^{-1}$ with our Haar-normalization (in particular in the normalization used in Sections~3--5 one often has $C=1$).  Using this identity with $u=t-s$ gives
\[
\sum_{\gamma\in\Gamma}\gamma(t-s) \;=\; C\cdot \mathbf{1}_{H}(t-s).
\]

Inserting this into \eqref{eq:Walnut-step1} yields
\[
(S_g f)(t)
= C\int_{K}f(s)\ \Big(\sum_{\lambda\in\Lambda} \overline{g(s-\lambda)}g(t-\lambda)\,\mathbf{1}_H(t-s)\Big)\,ds.
\]

Now make the change of variable $u=t-s\in H$, so $s=t-u$ and $ds=du$, we have
\[
(S_g f)(t) = C\sum_{\lambda\in\Lambda}\int_{u\in H}f(t-u)\ \overline{g(t-\lambda-u)}\,g(t-\lambda)\,du.
\]

Now we can write 
\[
(S_g f)(t) = \int_{H}G_u(t)f(t-u)du,
\]
where
\[
G_u(t)
\;:=\;
C\sum_{\lambda\in\Lambda}\ 
g(t-\lambda)\,\overline{g(t-\lambda-u)}.
\]

\end{proof}

\begin{remark}: 
Since the above Walnut representation is abstract, we therefore want to calculate $G_u(t)$ for a specific function $g$. The obvious choice is the characteristic function of a ball that is compactly supported and locally constant. The following proposition provides an explicit calculation.
\end{remark}
\qed

\begin{proposition}\label{prop:14}
   Let $K$ be a local field and $H=\mathfrak{D}$, $\Lambda$ and $\Gamma$ be a complete set of coset representatives for $K/H$ and $\widehat K/H^{\perp}$, respectively. Choose $g(t)={\bf 1}_{\mathfrak{p}^m\mathfrak{D}}(t)$. Then for all $f\in L^2(K)$, we have
   \[
   (S_g f)(t)= \int_{\mathfrak{p}^m\mathfrak{D}} f(t-s)ds.
   \]
   And in the Fourier domain
   \[
   \widehat{S_gf}(\xi) = {\bf 1}_{\mathfrak{p}^m\mathfrak{D}}(\xi)\widehat{f}(\xi).
   \]
   In particular, $\{M_{\gamma}T_{\lambda}g:\gamma\in\Gamma, \lambda\in\Lambda\}$ is a Parseval frame for the subspace
   $V_m=\{f:\text{supp}\hat{f}\subseteq \mathfrak{p}^m\mathfrak{D}\}.$

\end{proposition}
\begin{proof} By the previous theorem, for any $g\in L^2(K)$ such that the usual summability holds, the Gabor frame operator
\[
S_g f \;=\; 
\sum_{\lambda\in\Lambda}\sum_{\gamma\in\Gamma}
\langle f,\,M_\gamma T_\lambda g\rangle\, M_\gamma T_\lambda g
\]
has the Walnut form
\[
(S_g f)(t) \;=\; 
\int_H G_u(t)\, f(t-u)du,\qquad f\in L^2(K),
\]
where
\[
G_u(t)
\;:=\;
C\sum_{\lambda\in\Lambda}\ 
g(t-\lambda)\,\overline{g(t-\lambda-u)}.
\]

We need to compute now the Walnut's representation for $g={\bf 1}_{\mathfrak{p}^m\mathfrak{D}}$. Fix $t, u\in K$, we have
\begin{eqnarray*}
g(t-\lambda)g(t-\lambda-u) & = & {\bf 1}_{\mathfrak{p}^m\mathfrak{D}}(t-\lambda){\bf 1}_{\mathfrak{p}^m\mathfrak{D}}(t-\lambda-u)\\
& = & {\bf 1}_{\mathfrak{p}^m\mathfrak{D}}(u).
\end{eqnarray*}
The above equation holds because in case of local fields either two balls are disjoint or they properly contained in each other, therefore, $\mathfrak{p}^m\mathfrak{D}\cap (\mathfrak{p}^m\mathfrak{D}+u)=\mathfrak{p}^m\mathfrak{D}$ if $u\in\mathfrak{p}^m\mathfrak{D}.$ Hence, we have
\[
G_u(t)=\mathfrak{p}^m\mathfrak{D}(u), ~\forall t.
\]
This will give
\[
(S_g f)(t)= \int_{\mathfrak{p}^m\mathfrak{D}} f(t-u)du=f\ast{\bf 1}_{\mathfrak{p}^m\mathfrak{D}}(t).
\]

Hence, we have
\[
\widehat{S_gf}(\xi) = {\bf 1}_{\mathfrak{p}^m\mathfrak{D}}(\xi)\widehat{f}(\xi).
\]

The Gabor system $\{M_{\gamma}T_{\lambda}g:\gamma\in\Gamma, \lambda\in\Lambda\}$ is a Parseval frame for the subspace
   $V_m=\{f:\text{supp}\hat{f}\subseteq \mathfrak{p}^m\mathfrak{D}\}$ is direct from the above observation. Since $f\in V_m$,
   \[
   \widehat{S_gf}(\xi) = \widehat{f}(\xi),
   \]
   and hence $S_gf=f$, therefore,

   \[
   S_g f \;=\; 
\sum_{\lambda\in\Lambda}\sum_{\gamma\in\Gamma}
|\langle f,\,M_\gamma T_\lambda g\rangle\|^2=\langle Sf, f\rangle=\langle f,f\rangle=\|f\|^2
   \]
 which is Parseval's identity with bound $1$ on the subspace $V_m$. 
 \end{proof}

\begin{remark}
The Walnut representation established above provides a natural framework for
extending further results from the classical Euclidean Gabor theory to the
local-field setting. In particular, it suggests that key results based on
Walnut-type operator representations, such as those developed in \cite{CC}
and the method of double preconditioning described in \cite{FBH}, may admit
analogous formulations over non-Archimedean local fields. We do not pursue
these extensions here and leave them for future investigation.
\end{remark}

\subsection*{6.2. Wexler--Raz biorthogonality}
Before writing the local field version of the statement of the Wexler--Raz biorthogonality, we would like to write the formal definition of Feichtinger algebra on local fields $K$.
\begin{definition}[Feichtinger Algebra $S_0(K)$ on a Local Field]
Let $g_0 \in L^2(K)$ be a fixed nonzero window function, (e.g.\ $g_0 = \mathbf{1}_{\mathfrak D}$,
the characteristic function of the ring of integers in $K$).  
The \emph{Feichtinger algebra} $S_0(K)$ is defined by
\[
S_0(K) = \left\{ f \in L^2(K) : V_{g_0} f \in L^1(K \times \widehat K) \right\},
\]
equipped with the norm
\[
\| f \|_{S_0} = \| V_{g_0} f \|_{L^1(K \times \widehat K)}.
\]
This definition is independent of the particular choice of the window $g_0$,
up to equivalence of norms.
\end{definition}

The Feichtinger algebra $S_0(K)$ is continuously embedded into both $(L^1(K), \|\cdot\|_1)$ and $(C_0(K), \|\cdot\|_{\infty})$, hence into any $(L^p(K), \|\cdot\|_p)$, in particular into $(L^2(K), \|\cdot\|_2)$,
and is invariant under translation, modulation, and the Fourier transform.
It is the smallest Banach space invariant under time–frequency shifts,
and plays a fundamental role in Gabor analysis.

We now provide a local-field version of the Wexler--Raz biorthogonality relations for the co-compact lattice determined by $H=\mathfrak{P}^k$. 


\begin{theorem}[Wexler--Raz biorthogonality on a local field]
Let $K$ be a local field and $H$, $H^\perp$ and $\Lambda$, $\Gamma$ be as above. Let $g,h\in L^2(K)$ and assume that the
Gabor systems
\[
\mathcal G(g):=\{M_\gamma T_\lambda g : \lambda\in\Lambda,\ \gamma\in\Gamma\},
\qquad
\mathcal G(h):=\{M_\gamma T_\lambda h : \lambda\in\Lambda,\ \gamma\in\Gamma\}
\]
are Bessel sequences in $L^2(K)$ and form a pair of dual Gabor frames, i.e.,
\[
f=\sum_{\lambda\in\Lambda}\sum_{\gamma\in\Gamma}
\langle f,M_\gamma T_\lambda g\rangle\, M_\gamma T_\lambda h
\qquad \text{for all } f\in L^2(K),
\]
with convergence in $L^2(K)$.

Then, for every $\lambda\in \Lambda$ and every $\gamma\in \Gamma$,
\[
\langle h, M_\gamma T_\lambda g\rangle
=
\delta_{\lambda,0}\,\delta_{\gamma,0}.
\]
Equivalently,
\[
\int_K h(t)\,\overline{g(t-\lambda)}\,\overline{\gamma(t)}\,dt
=
\delta_{\lambda,0}\,\delta_{\gamma,0}.
\]
\end{theorem}

\begin{proof}

In the general LCA-group formulation, the Wexler--Raz relation is
\[
\langle h, M_\eta T_u g\rangle
=
d(\Delta)\,\delta_{u,0}\,\delta_{\eta,0},
\qquad (u,\eta)\in H\times H^\perp,
\]
where $d(\Delta)$ is the covolume (density factor) of the lattice $\Delta$.
For the lattice associated with $H=\mathfrak{P}^k$, this factor is
\[
d(\Delta)=\mu(H)\,\mu(H^\perp).
\]
By our Haar normalization, we have
\[
\mu(H)=|\mathfrak{P}^k|=q^{-k},
\qquad
\mu(H^\perp)=|\mathfrak{P}^{-k}|=q^{k},
\]
hence
\[
d(\Delta)=\mu(H)\mu(H^\perp)=q^{-k}q^{k}=1.
\]
Therefore the general Wexler--Raz identity reduces exactly to
\[
\langle h, M_\eta T_u g\rangle
=
\delta_{u,0}\,\delta_{\eta,0},
\qquad u\in H,\ \eta\in H^\perp,
\]
which proves the theorem.
\end{proof}

\begin{remark}[Converse under stronger hypotheses]
For the converse implication, i.e., 
\[
\langle h, M_\eta T_u g\rangle = \delta_{u,0}\,\delta_{\eta,0}
\quad \forall\, u\in H,\ \eta\in H^\perp,
\]
implies that $\mathcal G(g)$ and $\mathcal G(h)$ are dual Gabor frames. A standard sufficient condition is
$g,h\in S_0(K)$, in which the usual LCA-group proof applies to the present local field setting we refer \cite{feichtinger2009gabor, jakob}. 

\end{remark}

\subsection*{6.3. Janssen representation}

\begin{lemma}\label{lem:S0-implies-summability}
Let $K$ be a local field and $H$, $H^\perp$ and $\Lambda$, $\Gamma$ be as above. If $g\in S_0(K)$, then
\[
\sum_{\lambda\in\Lambda}\sum_{\gamma\in\Gamma}
\big|\langle g,\,M_\gamma T_\lambda g\rangle\big| < \infty.
\]
Equivalently,
\[
\sum_{\lambda\in\Lambda}\sum_{\gamma\in\Gamma}
|V_g g(\lambda,\gamma)| < \infty.
\]
In particular, the short-time Fourier transform
\[
(\lambda,\gamma)\mapsto V_g g(\lambda,\gamma)
\]
belongs to $\ell^1(\Lambda\times \Gamma)$.
\end{lemma}

\begin{proof}
Since $g\in S_0(K)$, the short-time Fourier transform $V_g g$ belongs to
$L^1(K\times \widehat K)$ and is continuous on $K\times \widehat K$.
Moreover, by the standard Wiener amalgam characterization of the Feichtinger
algebra (see, e.g., \cite{Grochenig01}), one has
\[
V_g g \in W(C_0,\ell^1)(K\times \widehat K).
\]
In particular,
\[
\sum_{\lambda\in\Lambda}\sum_{\gamma\in\Gamma}
\sup_{(x,\xi)\in (\lambda+H)\times(\gamma+H^\perp)}
|V_g g(x,\xi)| < \infty.
\]

For each $(\lambda,\gamma)\in \Lambda\times \Gamma$, define
\[
s_{\lambda,\gamma}
:=
\sup_{(x,\xi)\in (\lambda+H)\times(\gamma+H^\perp)}
|V_g g(x,\xi)|.
\]
Then
\[
\sum_{\lambda\in\Lambda}\sum_{\gamma\in\Gamma}
s_{\lambda,\gamma}<\infty.
\]
Since $(\lambda,\gamma)\in (\lambda+H)\times(\gamma+H^\perp)$ and $V_g g$ is
continuous, we have
\[
|V_g g(\lambda,\gamma)| \le s_{\lambda,\gamma},
\qquad (\lambda,\gamma)\in \Lambda\times \Gamma.
\]
Therefore,
\[
\sum_{\lambda\in\Lambda}\sum_{\gamma\in\Gamma}
|V_g g(\lambda,\gamma)|
\le
\sum_{\lambda\in\Lambda}\sum_{\gamma\in\Gamma}
s_{\lambda,\gamma}
<\infty.
\]
Finally, using the identity
\[
V_g g(\lambda,\gamma)=\langle g,M_\gamma T_\lambda g\rangle,
\]
we conclude that
\[
\sum_{\lambda\in\Lambda}\sum_{\gamma\in\Gamma}
\big|\langle g,M_\gamma T_\lambda g\rangle\big|<\infty.
\]
This proves the claim.
\end{proof}

\begin{theorem}[Janssen representation on a local field]\label{thm:janssen-coset}
Let $K$ be a local field and $H$, $H^\perp$ and $\Lambda$, $\Gamma$ be as above. Let $g\in L^2(K)$ and assume that the Gabor system $\mathcal{G}(g) = \{ M_\gamma T_\lambda g : \lambda \in \Lambda,\; \gamma \in \Gamma \}$ is a Bessel sequence in $L^2(K)$ and that
\[
\sum_{\lambda\in\Lambda}\sum_{\gamma\in\Gamma}
\big|\langle g, M_\gamma T_\lambda g\rangle\big| < \infty.
\]
Then the associated frame operator
\[
S_g f
  = \sum_{\lambda\in\Lambda}\sum_{\gamma\in\Gamma}
      \langle f, M_\gamma T_\lambda g\rangle\, M_\gamma T_\lambda g,
  \qquad f\in L^2(K),
\]
admits the \emph{Janssen representation}
\begin{equation}\label{eq:janssen-coset}
S_g
  = \sum_{\lambda\in\Lambda}\sum_{\gamma\in\Gamma}
      \langle g, M_\gamma T_\lambda g\rangle\, M_\gamma T_\lambda.
\end{equation}
The series in \eqref{eq:janssen-coset} converges unconditionally in the
strong operator topology on $L^2(K)$.
\end{theorem}

\begin{proof}
Set
\[
c_{\lambda,\gamma}:=\langle g,M_\gamma T_\lambda g\rangle,
\qquad (\lambda,\gamma)\in \Lambda\times\Gamma.
\]
By hypothesis,
\[
\sum_{\lambda\in\Lambda}\sum_{\gamma\in\Gamma}|c_{\lambda,\gamma}|<\infty.
\]
Hence, for every $f\in L^2(K)$,
\[
\sum_{\lambda\in\Lambda}\sum_{\gamma\in\Gamma}
\|c_{\lambda,\gamma} M_\gamma T_\lambda f\|
\le
\left(\sum_{\lambda\in\Lambda}\sum_{\gamma\in\Gamma}|c_{\lambda,\gamma}|\right)\|f\|_2
<\infty,
\]
since each $M_\gamma T_\lambda$ is unitary on $L^2(K)$.
Therefore the operator series
\[
\sum_{\lambda\in\Lambda}\sum_{\gamma\in\Gamma}
c_{\lambda,\gamma} M_\gamma T_\lambda
\]
converges unconditionally in the strong operator topology on $L^2(K)$.

The identification of this operator with the frame operator $S_g$ follows from the
standard Janssen representation for co-compact Gabor systems on locally compact abelian
groups, specialized to the lattice determined by the compact open subgroup
$H=\mathfrak P^k$ and its annihilator $H^\perp=\mathfrak P^{-k}$.
Note that in our case $\mu(H)\mu(H^\perp)=|\mathfrak P^k|\,|\mathfrak P^{-k}|=1,$ therefore additional covolume term will not appear similar to LCA group.
 Thus
\[
S_g
  = \sum_{\lambda\in\Lambda}\sum_{\gamma\in\Gamma}
      \langle g, M_\gamma T_\lambda g\rangle\, M_\gamma T_\lambda,
\]
as claimed.
\end{proof}

\begin{remark}
\begin{itemize}
\item[1.] The hypothesis
\[
\sum_{\lambda\in \Lambda}\sum_{\gamma\in \Gamma}
|\langle g, M_{\gamma} T_{\lambda} g\rangle| < \infty
\]
is the natural sufficient condition for the validity of the Janssen representation and is
strictly weaker than requiring $g\in S_0(K)$. In particular, if $g\in S_0(K)$, then
\[
\sum_{\lambda\in \Lambda}\sum_{\gamma\in \Gamma}
|V_g g(\lambda,\gamma)| < \infty,
\]
so the above hypothesis is automatically satisfied. Thus the Feichtinger algebra provides
a convenient and standard sufficient framework, but it is not necessary for the Janssen
representation itself.
\item[2.] 
In a general LCA group formulation, the Janssen representation is indexed by the adjoint lattice. In the present local field setting, we use the same notion $(\Lambda,\Gamma)$ for simplicity.

\item[3] The summability hypothesis described above is the local-field analogue of the condition in \cite{Walnut}. In the Euclidean case, under this condition, the associated Gabor frame operator extends boundedly to $L^p(\mathbb{R})$ for $1\le p\le \infty$. Thus, we can also see this as a natural analogue similar to Euclidean case.

\section*{Acknowledgements}

The authors would like to thank the anonymous referee for careful reading of the manuscript and for constructive comments and valuable suggestions, which helped us improve the clarity and overall presentation of the paper.

\end{itemize}

\end{remark}

\bibliographystyle{abbrv}
\bibliography{sn-bibliography}



\end{document}